\documentclass[10pt,leqno]{article}
\usepackage{fullpage}
\usepackage{amssymb}
\usepackage{amsmath}
\usepackage{amsthm}
\usepackage{longtable}
\newtheorem{theorem}{Theorem}

\newtheorem{corollary}{Corollary}

\newtheorem{lemma}{Lemma}
\newtheorem{remark}{Remark}
\begin{document}

\title{\Large A lower bound of Ruzsa's number related to the\\
Erd\H{o}s-Tur\'{a}n conjecture}
\author{\large Csaba S\'andor$^{1}$\footnote{
Email:~csandor@math.bme.hu. This author was supported by the OTKA
Grant No. K109789. This paper was supported by the J\'anos Bolyai
Research Scholarship of the Hungarian Academy of Sciences.}
~~Quan-Hui Yang$^{2}$\footnote{Email:~yangquanhui01@163.com. This
author was supported by the National Natural Science Foundation
for Youth of China, Grant No. 11501299, the Natural Science
Foundation of Jiangsu Province, Grant Nos. BK20150889,~15KJB110014
and the Startup Foundation for Introducing Talent of NUIST, Grant
No. 2014r029.} }
\date{} \maketitle
 \vskip -3cm
\begin{center}
\vskip -1cm { \small 1. Department of Stochastics, Budapest
University of Technology and Economics, H-1529 B.O. Box, Hungary}
 \end{center}

 \begin{center}
{ \small 2. School of Mathematics and Statistics, Nanjing University of Information \\
Science and Technology, Nanjing 210044, China }
 \end{center}

\begin{abstract} For a set $A\subseteq \mathbb{N}$ and $n\in \mathbb{N}$, let $R_A(n)$ denote the number of ordered pairs $(a,a')\in A\times A$ such that $a+a'=n$.  The
celebrated Erd\H{o}s-Tur\'{a}n conjecture says that, if $R_A(n)\ge
1$ for all sufficiently large integers $n$, then the
representation function $R_A(n)$ cannot be bounded. For any
positive integer $m$, Ruzsa's number $R_m$ is defined to be the
least positive integer $r$ such that there exists a set
$A\subseteq \mathbb{Z}_m$ with $1\le R_A(n)\le r$ for all $n\in
\mathbb{Z}_m$. In 2008, Chen proved that $R_{m}\le 288$ for all
positive integers $m$. In this paper, we prove that $R_m\ge 6$ for
all integers $m\ge 36$. We also determine all values of $R_m$ when
$m\le 35$.



\end{abstract}

\section{Introduction}

Let $\mathbb{N}$ be all nonnegative integers. For any set
$A,B\subseteq \mathbb{N}$, let
$$R_{A,B}(n)=\sharp \{(a,b):~ a\in A,~b\in B,~a+b=n\}. $$
Let $R_A(n)=R_{A,A}(n)$. If $R_A(n)\ge 1$ for all sufficiently
large integers $n$, then we say that $A$ is a basis of
$\mathbb{N}$. The celebrated Erd\H{o}s-Tur\'{a}n conjecture
\cite{erdosturan} states that if $A$ is a basis of $\mathbb{N}$,
then $R_A(n)$ cannot be bounded. Erd\H{o}s \cite{erdos} proved
that there exists a basis $A$ and two constants $c_1,c_2> 0$ such
that $c_1 \log n \leq R_A(n) \leq c_2 \log n$ for all sufficiently
large integers $n$. Recently, Dubickas \cite{Dubickas} gave the
explicit values of $c_1$ and $c_2$. In 2003, Nathanson
\cite{nathanson} proved that the Erd\H{o}s-Tur\'{a}n conjecture
does not hold on $\mathbb{Z}$. In fact, he proved that there
exists a set $A\subseteq \mathbb{Z}$ such that $1\leq R_A(n) \leq
2$ for all integers $n$. In the same year, Grekos et al.
\cite{Grekos} proved that if $R_A(n)\geq 1$ for all $n$, then
$\limsup_{n\rightarrow
 \infty}R_A(n)\geq 6.$ Later, Borwein et al. \cite{Borwein} improved 6 to
 8. In 2013, Konstantoulas \cite{Konstantoulas} proved that if the upper density
 $\overline{d}(\mathbb{N}\setminus (A+A))$ of the set of numbers
 not represented as sums of two numbers of $A$ is less than
 $1/10$, then $R_A(n)>5$ for infinitely many natural numbers $n$.
 Chen \cite{chen12} proved that there exists a basis $A$ of
 $\mathbb{N}$ such that the set of $n$ with $R_A(n)=2$ has density
 one. Later, the second author \cite{yang} and Tang \cite{tang15}
 generalized Chen's result. For the analogue of Erd\H{o}s-Tur\'an conjecture in groups, one can refer
 to \cite{Haddad},~\cite{Haddad08} and \cite{Konyagin}.

For a positive integer $m$, let $\mathbb{Z}_m$ be the set of
residue classes mod $m$. For $A,B\subseteq \mathbb{Z}_m$, let
$R_{A,B}(n)$ be the number of solutions of equation $a+b=n,~a\in
A,~b\in B$. Let $R_A(n)=R_{A,A}(n)$. If $R_A(n)\ge 1$ for all
$n\in \mathbb{Z}_m$, then $A$ is called an additive basis of
$\mathbb{Z}_m$.

In 1990, Ruzsa \cite{ruzsa} found a basis $A$ of $\mathbb{N}$ for
which $R_A(n)$ is bounded in the square mean. Ruzsa's method
implies that there exists a constant $C$ such that for any
positive integer $m$, there exists an additive basis $A$ of
$\mathbb{Z}_m$ with $R_A(n)\le C$ for all $n\in \mathbb{Z}_m$. For
each positive integer $m$, Chen \cite{chen08} defined Ruzsa's
number $R_m$ to be the least positive integer $r$ such that there
exists an additive basis $A$ of $\mathbb{Z}_m$ with $R_A(n)\le r$
for all $n\in \mathbb{Z}_m$. In this paper, Chen also proved that
$R_m\le 288$ for all positive integers $m$ and $R_{2p^2}\le 48$ for all prime numbers $p$. Until now, this is the
best upper bound about Ruzsa's number and there is no nontrivial
lower bound. In fact, in the same paper, Chen says ``We have
$R_m\ge 3$ for $m\not=1,2,3$. Now we cannot improve this trivial
lower bound".

In this paper, we give a nontrivial lower bound of Ruzsa's number.

%

\begin{theorem}\label{thm1}
$R_m=2$ if and only if $m=2,3$; $R_m=3$ if and only if $m=4,5,7$.
\end{theorem}

\begin{remark} If  $m>1$ and $A\subseteq \mathbb{Z}_m$ is an additive basis,
then $|A|\ge 2$. It follows that there exist two distinct elements
$a,a'\in A$, and so $R_A(a+a')\ge 2$. Hence $R_m=1$ if and only if
$m=1$.
\end{remark}

\begin{theorem}\label{thm2} $R_m=4$ if and only if
$m=6,8,9,10,11,12,13,14,15,19$; $R_m=5$ if and only if
$m=16,17,18,20,21,22,23,24,25,27,28,35$.
\end{theorem}

By Theorems \ref{thm1} and \ref{thm2}, we have the following
Corollary.

\begin{corollary} If $m\ge 36$, then $R_m\ge 6$.
\end{corollary}

\begin{remark} Furthermore, if $m\le 35$, then $R_m\le 6$. We list
all the values of $R_m~(2\le m\le 35)$ and a set $A \subseteq
\mathbb{Z}_m $ such that $1\le R_A(n)\le R_m$ for all $n\in
\mathbb{Z}_m$ in the Appendix.
\end{remark}

\section{Proofs}

In order to prove Theorems 1 and 2, we need some lemmas in the
following. The first lemma due to Lev and S\' ark\" ozy
\cite{levsarkozy} is the main tool of our proofs.

\begin{lemma}\label{lem1}{(Lev and S\' ark\" ozy's lower bound)} If $A$ is a subset
of a finite non-trivial abelian group $G$, then for any real
number $c$ we have
$$\sum_{g\in G}(R_A(g)-c)^2\ge
\frac{1}{|G|-1}\left(\frac{|A|^4}{|G|}-2|A|^3+|A|^2|G|\right).$$\end{lemma}

\begin{lemma}\label{lem2} Let $A\subseteq \mathbb{Z}_m$. If $R_A(n)\ge 1$ for all $n\in \mathbb{Z}_m$,
then $|A|>\sqrt{2m}-1/2$. \end{lemma}

\begin{proof}Since $R_A(n)\ge 1$ for all $n\in \mathbb{Z}_m$, we have
\begin{eqnarray*}
  |A|^2=\sum_{n=0}^{m-1}R_{A}(n)&\ge& |\{ n:n\in \mathbb{Z}_m, R_A(n)=1\}|+2|\{ n:n\in \mathbb{Z}_m, R_A(n)\ge 2\}|\\
 \nonumber &=&2|\{ n:n\in \mathbb{Z}_m\}|-|\{ n:n\in \mathbb{Z}_m,
  R_A(n)=1\}|\\
 \nonumber &=&2m-|\{ n:n\in \mathbb{Z}_m, R_A(n)=1\}| \ge 2m-|A|.
\end{eqnarray*}
Hence $(|A|+1/2)^2>2m$, that is,~$|A|>\sqrt{2m}-1/2$.
\end{proof}

\begin{lemma}\label{lem3} Let $A\subseteq \mathbb{Z}_m$ and $c$ be a positive integer.
If $R_A(n)\le c$ for all $n\in \mathbb{Z}_m$, then $|A|\le
\sqrt{cm}$.
\end{lemma}

This lemma follows from ~$|A|^2=\sum_{n=0}^{m-1}R_A(n)\le cm$
immediately.

\begin{lemma}(See \cite[P. 827, Test C]{Baumert}.)\label{lem4} Suppose that $v,\lambda,k~(v\ge k\ge \lambda)$ are positive integers.
Let $p$ be a prime divisor of $k-\lambda$ and let $w\ge 1$,
$(w,p)=1$, be a divisor of $v$ for which there exists an integer
$f>0$ such that $p^f\equiv -1~(\text{mod}~w)$. If $p^e$ exactly
divides $k-\lambda$ and $p^l~(l\ge 0)$ exactly divides $v$, then
there exists a set $A\subseteq \mathbb{Z}_v$ with $|A|=k$ such
that the congruence $a-a'\equiv b~(\text{mod}~v),~a,a'\in A$ has
exactly $\lambda$ distinct solutions for all $b\not\equiv
0~(\text{mod}~v)$ if and only if
$$p^{\lfloor e/2\rfloor}<(v/w)p^{-l},$$ where $\lfloor x \rfloor$
denotes the largest integer $\le x$.
\end{lemma}

\begin{lemma}\label{lem5} Let $A$ be an additive basis of ~$\mathbb{Z}_m$ and $k,l$ be positive integers with $(l,m)=1$. Then
$A+k,~lA$ is also an additive basis and
$$\max_{n\in
\mathbb{Z}_m}R_A(n)=\max_{n\in \mathbb{Z}_m}R_{A+k}(n)=\max_{n\in
\mathbb{Z}_m}R_{lA}(n).$$
\end{lemma}

 This lemma follows from $R_A(n)=R_{A+k}(n+2k)=R_{lA}(ln)$ for all $n\in
\mathbb{Z}_m$ immediately.


\begin{proof}[Proof of Theorem \ref{thm1}.] If $m\le 11$, by the computer-based
calculation, then we obtain that $R_m=2$ if and only if $m=2,3$
and $R_m=3$ if and only if $m=4,5,7$. Now it suffices to prove
that $R_m\le 3$ implies $m\le 11$. Suppose that $m\ge 12$ and
there exists a subset $A\subseteq \mathbb{Z}_m$ such that $1\le
R_A(n)\le 3$ for all $n\in \mathbb{Z}_m$.

Putting $G=\mathbb{Z}_m$ and $c=2$, by Lemma \ref{lem1}, we obtain
that for any subset $A\subseteq \mathbb{Z}_m$,
\begin{equation}\label{lower}
 \sum_{n=0}^{m-1}(R_A(n)-2)^2\ge \frac{|A|^2(m-|A|)^2}{m(m-1)}.
\end{equation}
Since $1\le R_A(n)\le 3$, it follows that
$$(R_A(n)-2)^2=
\begin{cases}
1, \text{ if $R_A(n)$ is odd};\\
0, \text{ if $R_A(n)$ is even}.
\end{cases}$$
Furthermore, if $R_A(n)$ is odd, then there exists $a\in A$ such
that $n=2a$, and so
\begin{equation}\label{upper}
\sum_{n=0}^{m-1}(R_A(n)-2)^2=\sum_{\substack{n=0\\2\nmid
R_A(n)}}^{m-1}1\le \sum_{a\in A}1=|A|.
\end{equation}
By (\ref{lower}) and (\ref{upper}), we have
\begin{equation*}\label{ine}
|A|(m-|A|)^2\le m(m-1)<m^2.
\end{equation*}

On the other hand, by Lemmas \ref{lem2} and \ref{lem3}, we have
~$\sqrt{2m}-1/2<|A|\le \sqrt{3m}.$ Hence
$$|A|(m-|A|)^2>
(\sqrt{2m}-1/2)(m-\sqrt{3m})^2>m^2,
$$
because $\sqrt{2m}-1/2>4$ and $\sqrt{3m}\le m/2$ for $m\ge 12$.
This is a contradiction.\end{proof}

\begin{proof}[Proof of Theorem \ref{thm2}.] We first prove that $R_m\le 5$ implies that $m\le 500$.
Suppose that $m>500$ and there exists $A\subseteq \mathbb{Z}_m$
such that~$1\le R_A(n)\le 5$ for all $n\in \mathbb{Z}_m$. By Lemma
\ref{lem1}, taking $G=\mathbb{Z}_m$ and $c=3$, we get
\begin{equation}\label{lower5}
 \sum_{n=0}^{m-1}(R_A(n)-3)^2\ge \frac{|A|^2(m-|A|)^2}{m(m-1)}.
\end{equation}

If $R_A(n)$ is odd, then~$(R_A(n)-3)^2\le 4$. If~$R_A(n)$ is even,
then~$(R_A(n)-3)^2=1$. Hence
\begin{eqnarray}\label{upper5}
&&\sum_{n=0}^{m-1}(R_A(n)-3)^2\\
\nonumber&\le& 4|\{ n:n\in \mathbb{Z}_m, R_A(n)\text{ is odd}\} |+|\{ n:n\in \mathbb{Z}_m, R_A(n)\text{ is even}\} |\\
\nonumber&=& m+3|\{ n:n\in \mathbb{Z}_m, R_A(n)\text{ is odd}\}
|\le m+3|A|.
\end{eqnarray}

By (\ref{lower5}) and (\ref{upper5}), we have
\begin{equation}\label{ine5}
  |A|^2(m-|A|)^2\le (m+3|A|)m(m-1).
\end{equation}

On the other hand, by Lemmas \ref{lem2} and \ref{lem3}, we
have~$\sqrt{2m}-1/2<|A|\le \sqrt{5m}$. Hence
\begin{eqnarray*}
|A|^2(m-|A|)^2&>&(\sqrt{2m}-1/2)^2(m-\sqrt{5m})^2>(1.9\cdot
0.9^2)m^3\\
&>&1.3m^3>(m+3\sqrt{5m})m^2>(m+3|A|)m(m-1),\end{eqnarray*}
because~$\sqrt{2m}-1/2>\sqrt{1.9m}$,~$m-\sqrt{5m}>0.9m$ and~
$m+3\sqrt{5m}<1.3m$ for~$m>500$. This contradicts with the
inequality \eqref{ine5}. Thus, if ~$m>500$, then~$R_m\ge 6$.

Now we only need to consider cases $m\le 500$.

If $m\le 20$, then the computer-based calculation can run over all
the sets ~$A\subseteq \mathbb{Z}_m$ with~$\sqrt{2m}-1/2\le |A|\le
\sqrt{5m}$ and we can determine these values of $R_m$. We obtain
that~$R_m=4$ for ~$m\in \{6,8,9,10,11,12,13,14,15,19\}$
and~$R_{16}=R_{17}=R_{18}=R_{20}=5$.


Next we assume that $21\le m\le 500$. A routine computer-based
calculation gives that the maximal pair of $(m,k)$ satisfying that
\begin{eqnarray}\label{cond1}21\le m\le 500, \quad \sqrt{2m}-1/2\le |A|=k\le \sqrt{5m}\end{eqnarray}
and the inequality (\ref{ine5}) holds is $(m,k)=(91,13)$. The
value for such $(m,k)$ is too large for the computer-based
calculation to run over all the sets $A\subseteq \mathbb{Z}_{91}$
with $|A|=13$.

In the following, we need three steps to reduce these values.

Our task is to find all exact pairs of $(m,k)$ with the following
property: There exists ~$A\subseteq \mathbb{Z}_m$ with~$|A|=k$
such that $1\le R_A(n)\le 5$ for all~$n\in \mathbb{Z}_m$. In the
first step, for $i\in \{1,2,3,4,5\}$, let
$$k_i=|\{ n: n\in \mathbb{Z}_m, R_A(n)=i\}|.$$
 Then
\begin{eqnarray}\label{cond2}
k_1+k_2+k_3+k_4+k_5=k,\quad k_i\in \mathbb{N}~(1\le i\le 5),
\end{eqnarray}
\begin{eqnarray}\label{cond3}
k^2=|A|^2=\sum_{n=0}^{m-1}R_A(n)=k_1+2k_2+3k_3+4k_4+5k_5,
\end{eqnarray}
and
\begin{eqnarray}\label{cond4}
k_1+k_3+k_5\le |A|=k,\quad \text{and the equality holds
when}~m~\text{is odd}.
\end{eqnarray}
By Lemma \ref{lem1}, taking $c=k^2/m$, we have
\begin{eqnarray}\label{cond5}
\sum_{n=0}^{m-1}\left(R_A(n)-\frac{k^2}{m}\right)^2=\sum_{i=1}^5\left(i-\frac{k^2}{m}\right)^2k_i\ge
\frac{|A|^2(m-|A|)^2}{m(m-1)}=\frac{k^2(m-k)^2}{m(m-1)}.\end{eqnarray}

By the computer-based calculation, the maximal values of $(m,k)$
such that there exists nonnegative integers $k_1,k_2,k_3,k_4,k_5$
satisfying \eqref{cond1}-\eqref{cond5} is $(50,12)$. This value is
also too large for the computer-based calculation to run over all
subsets $A\subseteq \mathbb{Z}_{50}$ with $|A|=12$.

In the second reduction step, we shall delete all pairs $(m,k)$
for which $42\le m\le 50$. Here we need to improve the Lev-S\'
ark\" ozy's bound. Clearly,
\begin{eqnarray}\label{eq101}
\sum_{n=0}^{m-1}\left(R_A(n)-\frac{k^2}{m}\right)^2&=&\sum_{n=0}^{m-1}R_A^2(n)-\frac{2k^2}{m}\sum_{n=0}^{m-1}R_A(n)+\frac{k^4}{m}\\
\nonumber&=&\sum_{n=0}^{m-1}R_A^2(n)-\frac{2k^2}{m}\cdot
k^2+\frac{k^4}{m}=\sum_{n=0}^{m-1}R_A^2(n)-\frac{k^4}{m}.
\end{eqnarray}
Next we use Lev-S\' ark\" ozy's arguments to obtain a better lower
bound for $\sum_{n=0}^{m-1}\left(R_A(n)-\frac{k^2}{m}\right)^2$.
Clearly, the sum~ $\sum_{n=0}^{m-1}R_A^2(n)$ counts the number of
solutions of the equation
$$a_1+a_2=a_3+a_4,\quad a_1,a_2,a_3,a_4\in A.$$
Rearranging these terms, one can rewrite this equation as
$a_1-a_3=a_4-a_2$. Hence
$$\sum_{n=0}^{m-1}R_A^2(n)=\sum_{n=0}^{m-1}R_{A,-A}^2(n)=k^2+\sum_{n=1}^{m-1}R_{A,-A}^2(n).$$
Clearly, $\sum_{n=1}^{m-1}R_{A,-A}^2(n)=k^2-k$. Let
$k^2-k=q(m-1)+r$, where $q,r$ are nonnegative integers and $0\le
r<m-1$. Then
$$q=\left\lfloor
\frac{k^2-k}{m-1}\right\rfloor\quad \text{and}\quad
r=k^2-k-\left\lfloor \frac{k^2-k}{m-1}\right\rfloor(m-1).$$ Hence
\begin{eqnarray}\label{eq102}&&\sum_{n=0}^{m-1}R_A^2(n)=k^2+\sum_{n=1}^{m-1}R_{A,-A}^2(n)\\
\nonumber&\ge& k^2+(q+1)^2r+q^2(m-1-r)\\
\nonumber&=&k^2+(2q+1)r+q^2(m-1)\\
\nonumber&=&k^2+\left(2\left\lfloor
\frac{k^2-k}{m-1}\right\rfloor+1\right)\left(k^2-k-\left\lfloor
\frac{k^2-k}{m-1}\right\rfloor(m-1)\right)+\left\lfloor
\frac{k^2-k}{m-1}\right\rfloor^2(m-1).\end{eqnarray}

By \eqref{cond5},~\eqref{eq101} and \eqref{eq102}, we get the
following better lower bound instead of \eqref{cond5}.
\begin{eqnarray}\label{cond6}\sum_{i=1}^5\left(i-\frac{k^2}{m}\right)^2k_i&\ge&
k^2+\left\lfloor
\frac{k^2-k}{m-1}\right\rfloor^2(m-1)-\frac{k^4}{m}\\
\nonumber&&+\left(2\left\lfloor
\frac{k^2-k}{m-1}\right\rfloor+1\right)\left(k^2-k-\left\lfloor
\frac{k^2-k}{m-1}\right\rfloor(m-1)\right).\end{eqnarray}

By the computer-based calculation, we list all pairs of $(m,k)$
such that there exist nonnegative integers $k_1,k_2,k_3,k_4,k_5$
satisfying \eqref{cond1}-\eqref{cond4} and \eqref{cond6} in the
following.
%
%
%

$(m,k)\in\{(21,7)$,$(21,8)$,$(21,9)$,$(22,7)$,$(22,8)$,$(22,9)$,$(23,7)$,$(23,8)$,
$(23,9)$,$(24,8)$,$(24,9)$,\\
$(25,8)$,$(25,9)$, $(26,8)$,$(26,9)$,$(27,8)$,$(27,9)$,$(28,8)$,
$(28,9)$,$(28,10)$,$(29,8)$,$(29,9)$,$(29,10)$,\\
$(30,9)$,$(30,10)$,$(31,9)$, $(31,10)$,
$(32,9)$,$(32,10)$,$(33,9)$,$(33,10)$,$(34,10)$,$(35,10)$,$(36,10)$,\\
$(36,11)$,$(37,11)$, $(38,11)$,
$(39,11)$,$(40,11)$,$(41,11)$,$(45,12)\}$.

In the last step, we deal with cases
$(m,k)=(40,11),(41,11),(45,12)$, since such values are also too
large for the computer-based calculation.

Now we first deal with the largest case $(m,k)=(45,12)$. Take
$v=45$,~$\lambda=3$,~$k=12$,~$p=3$,~$w=5$,~$f=2$,~$e=2$,~$l=2$. By
Lemma \ref{lem4}, it follows that there is no subset $A\subseteq
\mathbb{Z}_{45}$ with~ $|A|=12$ such that~ $R_{A,-A}(n)=3$ for all
$n\not\equiv 0~(\text{mod}~45)$. In other words, for any set
$A\subseteq \mathbb{Z}_{45}$, there exists $n\not\equiv
0~(\text{mod}~45)$ such that $R_{A,-A}(n)\not=3$. Noting that
~$\sum_{n=1}^{44}R_{A,-A}(n)=k^2-k=132$, we have
$$\sum_{n=1}^{44}R_{A,-A}^2(n)\ge 3^2\times 42+2^2+4^2=398.$$
Hence, by \eqref{eq101} and \eqref{eq102}, we have
$$\sum_{n=0}^{44}\left(R_A(n)-\frac{12^2}{45}\right)^2=12^2+\sum_{n=1}^{44}R_{A,-A}^2(n)-\frac{12^4}{45}\ge 81.2.$$

On the other hand, we list all values of $(k_1,k_2,k_3,k_4,k_5)$
when $(m,k)=(45,12)$ in the following.$\newline$

\begin{tabular}{|c|c|c|c|c|}
  \hline
  $k_1$ & $k_2$ & $k_3$ & $k_4$ & $k_5$ \\
  \hline
  $0$ & $24$ & $0$ & $9$ & $12$ \\
  $1$ & $22$ & $0$ & $11$ & $11$ \\
  $2$ & $20$ & $0$ & $13$ & $10$ \\
  $4$ & $16$ & $0$ & $17$ & $8$\\
  \hline
\end{tabular}~~\begin{tabular}{|c|c|c|c|c|}
  \hline
  $k_1$ & $k_2$ & $k_3$ & $k_4$ & $k_5$ \\
  \hline
   $5$ & $14$ & $0$ & $19$ & $7$ \\
   $6$ & $12$ & $0$ & $21$ & $6$ \\
  $7$ & $10$ & $0$ & $23$ & $5$ \\
   $8$ & $8$ & $0$ & $25$ & $4$\\
  \hline
\end{tabular}~~\begin{tabular}{|c|c|c|c|c|}
  \hline
  $k_1$ & $k_2$ & $k_3$ & $k_4$ & $k_5$ \\
  \hline
   $9$ & $6$ & $0$ & $27$ & $3$ \\
   $10$ & $4$ & $0$ & $29$ & $2$\\
  $11$ & $2$ & $0$ & $31$ & $1$ \\
   $12$ & $0$ & $0$ & $33$ & $0$ \\
  \hline
\end{tabular}$\newline$

For all the values list above, we have
$$\sum_{n=0}^{44}\left(R_A(n)-\frac{12^2}{45}\right)^2=\sum_{i=1}^5\left(i-\frac{12^2}{45}\right)^2k_i=79.2.$$
This is a contradiction.

Finally, we deal with the cases $(m,k)=(41,11)$ and $(40,11)$,
since the number of sets $A$ for which the computer-based
calculation can run over is about $39 \choose 9$.  If $m=41$, by
Lemma \ref{lem5}, then we can assume that $0,40\in A$. Hence the
number of such $A$ is $39 \choose 9$, and the computer-based
calculation can run over all such sets $A$. Now we consider the
case $m=40$. If there is an element in $A$ coprime with $40$, by
Lemma \ref{lem5}, then we can assume that $0,39\in A$, and so the
computer-based calculation can also deal with the case. If there
is no element in $A$ coprime with $40$, then we can assume that
$0\in A$ and
$$A\subseteq
\{0,2,4,5,6,8,10,12,14,15,16,18,20,22,24,25,26,28,30,32,34,35,36,38\}.$$
In this case, there are only $23 \choose 10$ sets $A$ and we can
deal with it with the computer-based calculation.

Using these idea, by the computer-based calculation, we obtain
$$R_m=4~ ~\text{if and only if}~~m=6,8,9,10,11,12,13,14,15,19;$$
$$R_m=5~~\text{if and only
if}~~m=16,17,18,20,21,22,23,24,25,27,28,35.$$
%

\end{proof}

\section{Appendix}

$$\begin{tabular}{|c|c|c|}
\hline
  $m$ & $R_m$ & the set $A$ \\
\hline
  $2$ & $2$ & $\{0,1\}$ \\
  \hline
  $3$ & $2$ & $\{0,1\}$ \\
  \hline
  $4$ & $3$ & $\{0,1,2\}$ \\
\hline
  $5$ & $3$ & $\{0,1,2\}$ \\
  \hline
  $6$ & $4$ & $\{0,3,4,5\}$ \\
  \hline
  $7$ & $3$ & $\{0,1,2,4\}$ \\
 \hline
  $8$ & $4$ & $\{0,3,5,6,7\}$ \\
  \hline
  $9$ & $4$ & $\{0,4,6,7,8\}$ \\
\hline
  $10$ & $4$ & $\{0,1,2,3,6\}$ \\
  \hline
  $11$ & $4$ & $\{0,4,6,8,9\}$ \\
  \hline
  $12$ & $4$ & $\{0,1,6,8,9,11\}$ \\
\hline
  $13$ & $4$ & $\{0,5,7,8,11,12\}$ \\
  \hline
  $14$ & $4$ & $\{0,4,8,9,11,12\}$ \\
  \hline
  $15$ & $4$ & $\{0,6,8,11,12,14\}$ \\
 \hline
  $16$ & $5$ & $\{0,1,2,3,4,7,11\}$ \\
  \hline
  $17$ & $5$ & $\{0,1,2,3,4,7,12\}$ \\
\hline
  $18$ & $5$ & $\{0,1,2,3,5,8,12\}$ \\
\hline
\end{tabular}~~\begin{tabular}{|c|c|c|}
\hline
  $m$ & $R_m$ & the set $A$ \\
  \hline
  $19$ & $4$ & $\{0,1,5,7,8,15,18\}$ \\
  \hline
  $20$ & $5$ & $\{0,1,2,5,6,13,16\}$ \\
 \hline
  $21$ & $5$ & $\{0,1,2,3,4,6,13,16\}$ \\
  \hline
  $22$ & $5$ & $\{0,1,2,4,5,9,15,17\}$ \\
  \hline
  $23$ & $5$ & $\{0,1,2,3,5,11,14,18\}$ \\
  \hline
  $24$ & $5$ & $\{0,1,2,6,9,10,12,17\}$ \\
  \hline
  $25$ & $5$ & $\{0,1,2,4,9,12,20,22\}$ \\
  \hline
  $26$ & $6$ & $\{0,1,2,5,15,19,20,22\}$ \\
  \hline
  $27$ & $5$ & $\{0,1,2,3,5,11,15,18,23\}$ \\
  \hline
  $28$ & $5$ & $\{0,1,2,4,5,8,10,17,22\}$ \\
  \hline
  $29$ & $6$ & $\{0,1,2,3,4,6,10,17,22\}$ \\
  \hline
  $30$ & $6$ & $\{0,1,2,3,4,5,7,11,17,22\}$ \\
  \hline
  $31$ & $6$ & $\{0,1,2,3,4,5,9,13,20,25\}$ \\
  \hline
  $32$ & $6$ & $\{0,1,2,3,4,5,8,15,20,26\}$ \\
  \hline
  $33$ & $6$ & $\{0,1,2,3,4,6,10,14,21,26\}$ \\
  \hline
  $34$ & $6$ & $\{0,1,2,3,4,6,13,19,26,29\}$ \\
  \hline
  $35$ & $5$ & $\{0,1,4,5,10,12,16,19,26,34\}$ \\
  \hline
\end{tabular}$$

\section{Acknowledgement} This work was done during the second
author visiting to Budapest University of Technology and
Economics. He would like to thank Dr. S\'andor Kiss and Dr. Csaba
S\'{a}ndor for their warm hospitality. He also would like to thank
Dr. Wenjun Cai for submitting his Matlab Program to a cluster of
computers with 64G memory.

\end{document}